\documentclass[runningheads]{llncs}

\usepackage{amsmath, amssymb, graphicx, microtype, hyperref, tikz-cd}
\usepackage{tikz}
\usepackage{soul}
\usepackage{mathtools}

\usepackage{comment, varwidth}
\usepackage{geometry}
\usepackage{subcaption}
\usepackage{bm}

\usepackage{todonotes}

\tikzcdset{scale cd/.style={every label/.append style={scale=#1},
    cells={nodes={scale=#1}}}}

\newcommand{\R}{\mathbb{R}}

\newcommand{\N}{\mathbb{N}}

\newcommand{\band}{\operatorname{Band}}

\newcommand{\s}{\operatorname{supp}}

\newtheorem{defi}{Definition}
\newtheorem{thm}{Theorem}
\newtheorem{lem}{Lemma}

\newtheorem{rmk}{Remark}

\title{On the Diameter of a 2-Sum of Polyhedra}

\titlerunning{Diameters of 2-Sums}
\date{}

\author{Steffen Borgwardt\inst{1} \and Weston Grewe\inst{1} \and Jon Lee\inst{2}}
\authorrunning{Borgwardt, Grewe, Lee}

\institute{University of Colorado Denver \and University of Michigan, Ann Arbor}

\begin{document}

\maketitle

\begin{abstract}
The study of the combinatorial diameter of a polyhedron is a classical topic in linear-programming theory due to its close connection with the possibility of a polynomial simplex-method pivot rule. The 2-sum operation is a classical operation for graphs, matrices, and matroids; we extend this definition to polyhedra. We analyze the diameters of 2-sum polyhedra, which are those polyhedra that arise from this operation. These polyhedra appear in matroid and integer-programming theory as a natural way to link two systems in a joint model with a single shared constraint and the 2-sum also appears as a key operation in Seymour's decomposition theorem for totally-unimodular matrices.  

We show that the diameter of a 2-sum polyhedron is quadratic in the diameters of its summands. The methods transfer to a linear bound for the addition of a unit column to an equality system, or equivalently, to the relaxation of an equality constraint to an inequality constraint. Further, we use our methods to analyze the distance between vertices on certain faces of a 3-sum polyhedron.
\end{abstract}


\section{Introduction}\label{sec:intro}

The \emph{(combinatorial) diameter} of a polyhedron is the minimum number of edges needed to form an edge walk between any pair of its vertices. The study of the diameter of a polyhedron  plays a central role in determining if there exists a pivot rule with polynomial running time for the simplex method (a longstanding open question in the theory of linear programming). Notably, the diameter is a lower bound on the number of iterations the simplex method may require, regardless of the choice of pivot rule. 

In general, tight upper bounds on the diameter are unknown. However, the diameter is at most quasi-polynomial in the number of facets $f$ and the dimension $\mathfrak{d}$, $f^{\log \mathfrak{d}+2}$ \cite{kk-92}, $(f - \mathfrak{d})^{\log \mathfrak{d}}$ \cite{t-14}, $(f - \mathfrak{d})^{\log O(\frac{\mathfrak{d}}{\log \mathfrak{d}})}$ \cite{s-19}. Additionally, the diameter of a polyhedron is polynomial in $\mathfrak{d}$ and the largest absolute value $\Delta$ over all subdeterminants of the constraint matrix: the diameter is at most $O(\Delta^2 \mathfrak{d}^4 \log(\mathfrak{d}\Delta))$ (and $O(\Delta^2 \mathfrak{d}^{3.5} \log(\mathfrak{d}\Delta))$ for bounded polytopes) \cite{bseh-12}.

Famously, Warren M. Hirsch asked if the diameter of a polyhedron with $f$ facets and dimension $\mathfrak{d}$ is at most $f - \mathfrak{d}$ \cite{d-63}; this statement is called the Hirsch conjecture. We call polyhedra whose diameter is bounded above by $f-\mathfrak{d}$ \emph{Hirsch satisfying}, and \emph{Hirsch violating} otherwise. The conjecture was first disproven for unbounded polyhedra \cite{kw-67} and much later for bounded polytopes \cite{s-11}. These counterexamples spurred interest into determining which classes of polyhedra are Hirsch satisfying. To date, several well-known classes of polytopes have been identified as Hirsch satisfying, including $0/1$-polytopes \cite{n-89},  network-flow polytopes \cite{bdf-17}, and the fractional stable-set polytope \cite{ms-14}.

In this paper, we study the diameters of $2$-sum polyhedra, which are those polyhedra whose constraint matrix can be decomposed as a $2$-sum of two matrices (Definition \ref{def:matrix_2sum}). Such polyhedra are encountered in both the theory and the application of linear and integer programming. 

The $2$-sum is a classical matroid operation and the matrix definition arises when applying the $2$-sum to represented matroids. This operation is used in the decomposition of regular matroids into graphic matroids, co-graphic matroids, and a special $10$-element matroid. Equivalently, this decomposition can be modified to decompose totally-unimodular matrices into network matrices and two special $5 \times 5$ matrices \cite{s-98,s-80}.

We also encounter the $2$-sum in time-staged decision making. Here, the operation represents the situation where the decision at stage $t+1$ depends on the decision at stage $t$ through a single shared resource. Time-staged decision making is a special case of ``staircase linear programming,'' which is known to be challenging for the classical simplex method. The $2$-sum operation has also been used to devise an efficient algorithm for solving bimodular integer programs, the special case of integer programming where the subdeterminants of the constraint matrix are contained in $\{ 0, \pm 1, \pm 2 \}$ \cite{awz-17}.

The $2$-sum operation is among the simplest of operations to join two polyhedra for which a good bound on the resulting diameter is unknown. For example, the diameter of a Cartesian product of two polyhedra is the sum of the individual diameters, $d(P \times Q) = d(P) + d(Q)$, and the diameter of the parallel-connection of two polyhedra, an operation from which the $2$-sum can be derived, is at most the sum of the diameters plus two \cite{bgl-23}. We will prove that the diameter of the $2$-sum of polyhedra is at most quadratic in the diameters of its inputs. 

We will construct an upper bound on the diameter of a $2$-sum polyhedron $P$ that depends on the diameters of polyhedra whose $2$-sum is $P$. Critically, the $2$-sum operation for polyhedra is not injective; there is a degree of freedom with respect to the right-hand side (see Subsection \ref{subsec:terminology}). For this reason, we phrase our bounds in terms of diameters over the \emph{class} of standard-form polyhedra that share the same equality constraint matrix but may have differing right-hand sides. This is not necessarily a limitation. For many classes of polyhedra (e.g., network-flow, dual-transportation), tight bounds on the diameter are known for the class, but not for the individual polyhedra in the class. 

TU polyhedra are those polyhedra that can be described with a totally-unimodular constraint matrix. The diameter of a $\mathfrak{d}$-dimensional TU polyhedron is at most $O(\mathfrak{d}^4\log (\mathfrak{d}))$ because $\Delta \leq 1$ for a TU matrix \cite{bseh-12}. In fact, many of the most well-studied classes of TU polyhedra are actually Hirsch satisfying, such as, network flow polytopes, assignment polytopes, and partition polytopes \cite{br-74,b-13,bdf-17}; however, it is a longstanding open question whether all TU polyhedra are Hirsch satisfying. In Section \ref{sec:2sum_bound}, we discuss an approach to reduce the proven quadratic diameter bound for the $2$-sum to a linear diameter bound. Proving that the diameter of the $2$-sum polyhedron is linear in the diameter of the inputs should be a key step in deriving a tighter diameter bound for TU polyhedra, or possibly showing that TU polyhedra are Hirsch satisfying.

\subsection{Outline}

In the remainder of Section \ref{sec:intro}, we formally define the {\em $2$-sum of polyhedra} and the {\em band of a vertex}, which will be a useful tool for proving our results. Throughout, we are interested in ``lifting'' walks to the $2$-sum from its summands. In Lemma  \ref{lem:band_step_criterion}, we prove that the band of a vertex yields a criterion for determining when a walk can be lifted.

We begin Section \ref{sec:2sum_bound} by discussing some immediate observations on the lifting of walks between polyhedra and derive bounds on the distance between some pairs of vertices of the $2$-sum. We use these observations for a constructive proof of a quadratic diameter bound for a $2$-sum polyhedron (Theorem \ref{thm:2sum_diam_bound}). Following the proof, we discuss how our bound possibly can be improved to be linear. Finally, we consider the addition of a unit column to the constraint matrix, which is equivalent to relaxing one equality constraint to inequality. We show that this operation can be interpreted as a special $2$-sum, and in this special case, we derive a linear bound on the resulting diameter. 

In Section \ref{sec:3sum}, we study how our methods can transfer to $3$-sum polyhedra, which we define as those polyhedra that can be represented with an equality constraint matrix that is a $3$-sum of two matrices. We discuss which of our tools and results transfer to $3$-sum polyhedra and which do not. We conclude with a brief outlook on some remaining open questions, in Section \ref{sec:outlook}.

\subsection{Notation and Terminology}\label{subsec:terminology}

Given two vertices $v, w \in P$, the distance between $v$ and $w$, $d_P(v, w)$, is the minimum number of edges needed to construct an edge walk from $v$ to $w$; we drop the subscripted $P$ when the polyhedron is clear from context. The diameter of $P$, $d(P)$, is the maximum of $d_P(v, w)$ over all pairs of vertices. 

It will be convenient to discuss the diameter of the class of polyhedra that share a constraint matrix but have different right-hand sides. Given a matrix $A \in \R^{m,n}$, we set $P(b) := \{ x : Ax = b, x \geq 0 \}$ and define $d(A) := \max \{ d(P(b)) : b \in \R^m \}$. Thus, the value $d(A)$ is the largest diameter of any standard-form polyhedron with constraint matrix $A$.

For $P = \{ x : Ax=b, x \geq 0\}$ the value $d(P)$ may be considerably smaller than $d(A)$, especially for a degenerate polyhedron. The value $d(A)$ is commonly studied when bounding the diameter of a \emph{class} of polyhedra with varying right-hand sides. For example, when $A$ is the node-arc incidence matrix of a digraph $G$, $d(A)$ is an upper bound on the diameter of any uncapacitated network-flow polytope with underlying graph $G$. Our bound for the diameter of a $2$-sum polyhedron is particularly useful when the polyhedron can be decomposed into polyhedra that belong to classes with well-understood diameters, e.g., a network-flow polytope and a dual-transportation polytope. In this case, we have that the diameter is at most quadratic in its associated Hirsch bound.

 We follow \cite{s-98} to define the $2$-sum operation for a pair of arbitrary matrices.

\begin{defi}[2-Sum of Matrices]\label{def:matrix_2sum}
Given two matrices $\begin{bmatrix} A & a \end{bmatrix}$ and $\begin{bmatrix}
        b \\
        B
    \end{bmatrix}$ with distinguished column $a$ and row $b$. The $2$-sum is
    \[
    \begin{bmatrix} 
        A & a 
    \end{bmatrix}
    \oplus_2
    \begin{bmatrix}
        b \\
        B
    \end{bmatrix}
    =
    \begin{bmatrix}
        A & ab \\
        0 & B \\
    \end{bmatrix}.
    \]
\end{defi}

The $2$-sum operation was designed to generalize its graph-theoretical case, which corresponds to a ``clique sum'' on a clique of size 2. Consider a pair of graphs $G_1$ and $G_2$, with a unique edge $e$ that lies in both graphs. To form the $2$-sum, we assign $e$ an orientation and glue $G_1$ and $G_2$ together along edge $e$, respecting the orientation, then we delete edge $e$. An example is given in Figure \ref{fig:2-sum_graph}.

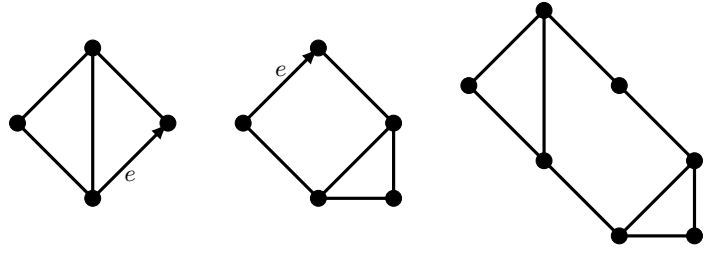
\begin{figure}
    \centering
    \begin{tikzpicture}[scale=1]

        \draw[very thick, black] (1,2)--(0,1)--(1,0)--(1,2)--(2,1);
        \draw[very thick, black, -latex] (1,0)--node[below] {$e$}(2,1);

        \draw[fill = black] (1, 0) circle (3pt);
        \draw[fill = black] (1, 2) circle (3pt);
        \draw[fill = black] (0, 1) circle (3pt);
        \draw[fill = black] (2, 1) circle (3pt);

        \draw[very thick, black] (4,2)--(5,1)--(5,0)--(4,0)--(3,1);
        \draw[very thick, black] (4,0)--(5,1);
        \draw[very thick, black, -latex] (3,1) --node[above] {$e$} (4,2);

        \draw[fill = black] (4, 2) circle (3pt);
        \draw[fill = black] (5, 1) circle (3pt);
        \draw[fill = black] (5, 0) circle (3pt);
        \draw[fill = black] (4, 0) circle (3pt);
        \draw[fill = black] (3, 1) circle (3pt);

        \draw[very thick, black] (6,1.5)--(7,2.5)--(8,1.5)--(9,0.5)--(9,-0.5)--(8,-0.5)--(7,0.5)--(6,1.5);
        \draw[very thick, black] (7,0.5)--(7,2.5);
        \draw[very thick, black] (8,-0.5)--(9,0.5);

        \draw[fill = black] (6, 1.5) circle (3pt);
        \draw[fill = black] (7, 2.5) circle (3pt);
        \draw[fill = black] (8, 1.5) circle (3pt);
        \draw[fill = black] (9, 0.5) circle (3pt);
        \draw[fill = black] (9, -0.5) circle (3pt);
        \draw[fill = black] (8, -0.5) circle (3pt);
        \draw[fill = black] (7, 0.5) circle (3pt);
    \end{tikzpicture}
    \caption{Graphs $G_1$ (left) and $G_2$ (center) with unique, oriented edge $e$, and $G_1 \oplus_2 G_2$ along edge $e$ (right).}
    \label{fig:2-sum_graph}
\end{figure}

We define the $2$-sum for standard-form polyhedra by performing the operation on the constraint matrices and making an adjustment to the right-hand side. \newpage 
\begin{defi} \label{def:2sum}
Given two standard-form polyhedra
\[
    P := \left\{ x : \begin{bmatrix}
        A & a\\
    \end{bmatrix}
    x
    =
    c_A, \, 
    x \geq 0
    \right\}, \qquad
    Q := \left\{ y : \begin{bmatrix}
        b \\
        B
    \end{bmatrix}
    y
    =
    \begin{bmatrix}
        c_b \\
        c_B
    \end{bmatrix}, \, 
    y \geq 0
    \right\},
\]
the $2$-sum of $P$ and $Q$ is 
\[
    P \oplus_2 Q :=  \left\{ \begin{bmatrix}
        x \\
        y
    \end{bmatrix} : \begin{bmatrix}
        A & ab \\
        0 & B \\
    \end{bmatrix}
    \begin{bmatrix}
        x \\
        y
    \end{bmatrix}
    =
    \begin{bmatrix}
        c_A + ac_b \\
        c_B
    \end{bmatrix}, \, 
    x, y \geq 0
    \right\}.
\]
\end{defi}

When transferring the definition of the $2$-sum to polyhedra, there is a degree of freedom for the choice of right-hand side. In Definition \ref{def:2sum}, our choice of right-hand side implies that if $(x, 0) \in P$ and $y \in Q$, then $(x, y) \in P \oplus_2 Q$. When decomposing an arbitrary $2$-sum polyhedron, one can choose the right-sides for $P$ and $Q$ to match our definition.

Throughout, we assume that $a \neq 0$ and $b \neq 0$. For a zero-column $a$, the polyhedron would have 
(half of) a lineality space and 
all vertices set the corresponding variable to $0$.
For a zero-row $b$, the constraint $by = c_b$ is either redundant or makes the associated system infeasible. 
Because $a \neq 0$ and $b \neq 0$, we have $ab \neq 0$. When $ab \neq 0$, $P \oplus_2 Q$ can be represented in the following form by applying row operations to the equality constraints. We have
\begin{equation} \label{eqn:2sum_reduced}
     P \oplus_2 Q =
     \left\{ \begin{bmatrix} x \\
        y
    \end{bmatrix} : \begin{bmatrix}
        A' & 0 \\
        a' & b \\
        0 & B \\
    \end{bmatrix}
    \begin{bmatrix}
        x \\
        y
    \end{bmatrix}
    =
    \begin{bmatrix}
        c_A' \\
        c \\
        c_B
    \end{bmatrix}, \, 
    x, y \geq 0
    \right\},
\end{equation}
for some $A', a', c'_A,$ and $c$. Thus, every $2$-sum polyhedron can be decomposed as the $2$-sum of polyhedra with the following forms

\begin{align} \label{eqn:2sum_decomp}
    P &:= \left\{ x : \begin{bmatrix}
        A' & 0\\
        a' & 1
    \end{bmatrix}
    \begin{bmatrix}
        x \\
        s
    \end{bmatrix}
    =
    \begin{bmatrix}
        c'_A \\
        c_{a'}
    \end{bmatrix}, \, 
    x, s \geq 0
    \right\}, \quad
    Q := \left\{ y : \begin{bmatrix}
        b \\
        B
    \end{bmatrix}
    y
    =
    \begin{bmatrix}
        c_b \\
        c_B
    \end{bmatrix}, \, 
    y \geq 0
    \right\},
 \end{align}
where $c = c_{a'} + c_b$. 

Throughout, we assume $P \oplus_2 Q$ is represented in the form of Equation \ref{eqn:2sum_reduced}. For notational convenience, we drop the apostrophes from $A'$, $a'$, $c'_A$ in Equation \ref{eqn:2sum_reduced} and set $\bar{A} = \binom{A}{a}$ and $\bar{B} = \binom{b}{B}$. We will extensively use the following two polyhedra 
\[
P_A := \{ x : Ax = c_A, x \geq 0 \}, \quad 
Q_B := \{ y : By = c_B, y \geq 0 \},
\] 
which drop the constraint rows for $a'$ and $b$, respectively.
We also refer to variants of (a face of) $P$ and $Q$ where the right-hand side is adjusted as such:
\begin{align*}
    P(y) &:= \left\{ x : \begin{bmatrix}
        A\\
        a
    \end{bmatrix}
    x
    =
    \begin{bmatrix}
        c_A \\
        c_a + c_b - by
    \end{bmatrix}, \, 
    x \geq 0
    \right\}, \\
    Q(x) &:= \left\{ y : \begin{bmatrix}
        b \\
        B
    \end{bmatrix}
    y
    =
    \begin{bmatrix}
        c_a + c_b - ax \\
        c_B
    \end{bmatrix}, \, 
    y \geq 0
    \right\}.
\end{align*}
Intuitively, for a given $y \in Q_B$, $P(y)$ is the set of $x$ such that $(x, y) \in P \oplus_2 Q$, and vice versa for a given $x \in P_A$ and $Q(x)$. 

We are interested in ``lifting'' walks in $P_A$, $Q_B$, $P(y)$, and $Q(x)$ to $P \oplus_2 Q$. That is, we show that in some cases, walks in the input polyhedra correspond to walks in $P \oplus_2 Q$. Recall that the \emph{support} of a vector is the set indexing its nonzero components. We define the lift of a walk formally as follows.

\begin{defi} \label{def:lift}
Suppose that $(x^1, y^1) \in P \oplus_2 Q$ is a vertex, and $x^1$ is a vertex of $P_A$. Let $x^1, \ldots, x^k \in P_A$ a sequence of vertices such that $x^i$ is adjacent to $x^{i+1}$ for $1 \leq i \leq k-1$. We say the walk $x^1, \ldots, x^k$ \textit{lifts} to $P \oplus_2 Q$ if, for each $i$, there exists $y^i$ such that $(x^i, y^i) \in P \oplus_2 Q$, and $y^i$ has the same support as $y^1$.
\end{defi}

If a path in $P_A$ lifts to $P \oplus_2 Q$, then the resulting sequence of vertices $(x^1, y^1), \ldots, (x^k, y^k)$ is an edge walk in $P \oplus_2 Q$. To see this, we note that $P_A$ is affinely isomorphic to the face of $P \oplus_2 Q$ formed by requiring the constraints $y_i \geq 0$ to be satisfied with equality for each $i$ such that $y_i^1 = 0$. Further, when $P \oplus_2 Q$ is simple, $\s(y^i) = \s(y^1)$ for each $1 \leq i \leq k$. If instead $\s(y^i) \subsetneq \s(y^1)$, then $(x^i, y^i)$ is a degenerate vertex, which contradicts the assumption of simplicity.

Often, we are interested in determining if a single step of a walk lifts. When the step does, we say that the step \emph{lifts successfully}. Otherwise, we say that the step \emph{fails to lift}. In either case, when the lift of a step to $(x,y)$ results in $(x',y')$, we say the step \emph{terminates} at $(x',y')$.

Indeed, if for every pair of vertices in $P \oplus_2 Q$, a path between the two vertices can be constructed via lifted paths, then the diameter of $P \oplus_2 Q$ can be bounded by a function of the diameters of $P_A$, $Q_B$, $P(y)$, and $Q(x)$ for some $x$ and $y$.

\subsection{The Band of a Vertex}

We seek to determine an upper bound on $d(P \oplus_2 Q)$ that depends solely on the constraint matrix of $P \oplus_2 Q$. We may assume $P \oplus_2 Q$ is simple; if $P \oplus_2 Q$ is not simple, then there exists a perturbation of the right-hand side such that the resulting polyhedron is simple and has diameter at least $d(P \oplus_2 Q)$ \cite{ykk-84}.
When $P \oplus_2 Q$ is simple, $P(y)$ and $Q(x)$ also are simple because they are faces of $P \oplus_2 Q$. We have a categorization of the vertices of $P \oplus_2 Q$. If $(x, y) \in P \oplus_2 Q$ is a vertex, then either $x \in P(y)$ is a vertex and $y$ lies on an edge of $Q_B$, or vice versa. If $x \in P_A$ is a vertex, we call $(x, y)$ an \emph{$\mathbf{x}$-vertex}. Otherwise, $y \in Q_B$ is a vertex and we call $(x, y)$ a \emph{$\mathbf{y}$-vertex}.

As stated above, when $(x, y)$ is an $\mathbf{x}$-vertex, $y \in Q_B$ lies on an edge of $Q_B$. This gives a ``degree of freedom'' that we will use to determine when a path in $P_A$ can be lifted to $P \oplus_2 Q$. We will call this degree of freedom a \textit{band} (Definition \ref{def:band}). Recall that the support of a vector $y$, $\s(y)$, is the set of indices such that $y_i \neq 0$.

\begin{defi} \label{def:band}
Let $(x^0, y^0) \in P \oplus_2 Q$ be an $\mathbf{x}$-vertex. The $\mathbf{x}$-band (with respect to $y^0$) is the set of values given by the following set:

\[
    \band_{\mathbf{x}}(y^0) := \{ c_a + c_b - bz : Bz = c_B, 
     \s(z) \subseteq \s(y^0), \, z \geq 0 \}.
\]
\end{defi}

We can define the $\mathbf{y}$-band for a $\mathbf{y}$-vertex analogously. Lemma \ref{lem:band_step_criterion} shows that the set $\band_{\mathbf{x}}(y)$ can be used to determine when a step lifts.

\begin{lem} \label{lem:band_step_criterion}
Let $(x, y) \in P \oplus_2 Q$ be an $\mathbf{x}$-vertex. Suppose that $x' \in P_A$ is an adjacent vertex to $x$. The step $x$ to $x'$ lifts to $P \oplus_2 Q$ if and only if $ax' \in \band_{\mathbf{x}}(y)$.
\end{lem}

\begin{proof}
If the step $x$ to $x'$ lifts, then there exists $y'$ such that $(x', y') \in P \oplus_2 Q$ and, by simplicity of $P \oplus_2 Q$, $y'$ has the same support as $y$. Also, $By' = c_B$ and $y' \geq 0$ because $(x', y') \in P \oplus_2 Q$. It follows that $ax' = c_a + c_b - by' \in \band_{\mathbf{x}}(y)$.

Now, suppose that $ax' \in \band_{\mathbf{x}}(y)$. By definition of $\band_{\mathbf{x}}(y)$, there exists some $y'$ such that $By' = c_B$, $\s(y') \subseteq \s(y)$, $y \geq 0$, and $ax' = c_a + c_b - by'$. Therefore, $(x', y') \in P \oplus_2 Q$. By simplicity of $P \oplus_2 Q$, the point $(x', y')$ is a vertex---the existence of a feasible point with support strictly contained in the support of $(x', y')$ would imply the existence of a degenerate vertex. Finally, $(x', y')$ and $(x, y)$ are adjacent because $x$ and $x'$ are adjacent in $P_A$ and $(x, y)$, $(x', y')$ are contained on a face of $P \oplus_2 Q$ isomorphic to $P_A$.
\qed
\end{proof}

Lemma \ref{lem:band_step_criterion} can be reformulated to state that if $(x, y)$ is an $\mathbf{x}$-vertex and $x' \in P_A$ is adjacent to $x$, then there exists a $y'$ with $\s(y') \subseteq \s(y)$, such that $(x', y') \in P \oplus_2 Q$ is an adjacent vertex of $(x, y)$ if and only if $ax' \in \band_{\mathbf{x}}(y)$.

Consider an $\mathbf{x}$-vertex $(x, y)$ and a vertex $x' \in P_A$ adjacent to $x$. If the step from $x$ to $x'$ fails to lift to $(x, y)$, we use the fact that $y$ is on an edge of $Q_B$ to determine the vertex at which the step terminates.

\begin{lem} \label{lem:out_of_band}
If $(x, y) \in P \oplus_2 Q$ is an $\mathbf{x}$-vertex, then $y$ lies on an edge of $Q_B$, denoted $\operatorname{conv}(y^1, y^2)$. If vertex $x^1 \in P_A$ is adjacent to $x$ and the step $x$ to $x^1$ fails to lift to $P \oplus_2 Q$, then the resulting vertex is $(x', y^i)$ for some $x' \in \operatorname{conv}(x, x^1)$ and $i \in \{ 1, 2 \}$.
\end{lem}

\begin{proof}
We claim $y \in Q(x)$ is a vertex. Suppose that $y$ is not a vertex. Then, there exists a vertex $\bar{y} \in Q(x)$ with $\s(\bar{y}) \subsetneq \s(y)$. Further, $(x, \bar{y}) \in P \oplus_2 Q$, which contradicts that $(x, y)$ is a vertex of $P \oplus_2 Q$.

Let $I$ index the support of $y$, let $\bar{B}_I$ (and $B_I$) be the submatrix of $\bar{B}$ (and $B$) with columns indexed by $I$, respectively. Suppose that $\bar{B}_I$ has $m$ rows; by irredundancy, $\operatorname{rank}(\bar{B}_I) = m$. Because $B_I$ is obtained by deleting a single row of $\bar{B}_I$, it follows from irredundancy that $\operatorname{rank}(B_I) = m-1$, and thus, $\operatorname{ker}(B_I)$ is a $1$-dimensional subspace. Consequently, $y$ lies on the edge specified by $\{ z : B_Iz = c_B, z \geq 0 \} = \operatorname{conv}(y^1, y^2)$ where $y^1 := \operatorname{argmin} \{ bz : B_Iz = c_B, z \geq 0 \}$ and $y^2 := \operatorname{argmax} \{ bz : B_Iz = c_B, z \geq 0 \}$.

Suppose that $x^1 \in P_A$ is an adjacent vertex of $x$ and the step $x$ to $x^1$ fails to lift. By Lemma \ref{lem:band_step_criterion}, $ax^1 \notin \band_{\mathbf{x}}(y)$. If $ax^1 > c_a + c_b - by^1$, then lifting the step will terminate at some vertex $(x', y^1)$. It follows that $y^1 \in Q_B$ is a vertex, hence $(x', y^1)$ is a $\mathbf{y}$-vertex. Similarly, if $ax^1 < c_a + c_b - by^2$ then $(x', y^2)$ is a $\mathbf{y}$-vertex.
\qed
\end{proof}

We use Lemmas \ref{lem:band_step_criterion} and \ref{lem:out_of_band} to determine when walks in some input polyhedra lift to walks in the $2$-sum. In turn, we establish a diameter bound for $P \oplus_2 Q$ that depends on the diameter of the inputs.

\section{A Quadratic Diameter Bound for the 2-Sum}\label{sec:2sum_bound}

We present a quadratic diameter bound for $P \oplus_2 Q$. Before we prove the main theorem, we first prove three lemmas. Each lemma gives distance bounds between a subset of the pairs of vertices. The bound for $d(P \oplus_2 Q)$ is presented in Theorem \ref{thm:2sum_diam_bound}. 

First, we show that if two $\mathbf{x}$-vertices have the same $x$-coordinates, then the distance between the two vertices is at most $d(\bar{B})$.

\begin{lem}\label{lem:m-path}
    If $(x,y), (x,y') \in P \oplus_2 Q$ are both $\mathbf{x}$-vertices, then 
    \[
        d((x,y), (x,y')) \leq d(\bar{B}).
    \]
\end{lem}

\begin{proof}
From the proof of Lemma \ref{lem:out_of_band}, both $y, y'$ are vertices of $Q(x)$. Let $y=y^0$, $y^1$, $\ldots$, $y^k=y'$ denote a shortest path between $y$ and $y'$ in $Q(x)$. It follows that $k \leq d(Q(x)) \leq d(\bar{B})$. We claim that for each $0 \leq i \leq k$, $(x, y^i) \in P \oplus_2 Q$ is a vertex. It is easy to check that $(x, y^i) \in P \oplus_2 Q$. Further, if $(x, y^i)$ is not a vertex, then there exists $(\hat{x}, \hat{y})$ with $\s(\hat{x}, \hat{y}) \subsetneq \s (x, y^i)$. However, $x \in P_A$ is a vertex, so $\hat{x} = x$. Therefore, $\s(\hat{y}) \subsetneq \s(y^i)$, and thus $y^i$ is not a vertex; a contradiction. We note that for each $0 \leq i < k$, the vertices $(x, y^i)$ and $(x, y^{i+1})$ are adjacent, which follows from the adjacency of $y^i$ and $y^{i+1}$ in $Q$. We conclude that $d((x,y), (x,y')) \leq k \leq d(\bar{B})$. 
\qed
\end{proof}

The proof can be readily adapted to show that if $(x, y)$ and $(x', y)$ are two $\mathbf{y}$-vertices then $d((x,y), (x',y)) \leq d(\bar{A})$. Next, we show that a $\mathbf{y}$-vertex is always ``close'' to an $\mathbf{x}$-vertex.

\begin{lem} \label{lem:change_cat}
Suppose that $P \oplus_2 Q$ contains both $\mathbf{x}$- and $\mathbf{y}$-vertices. If $(x^1, y^1)$ is a $\mathbf{y}$-vertex, then there exists an $\mathbf{x}$-vertex, $(x',y')$, such that 
\[
d((x^1, y^1), (x', y')) \leq d(\bar{A}) + d(Q_B) + 1.
\]
\end{lem}

\begin{proof}
Because $P \oplus_2 Q$ contains both $\mathbf{x}$- and $\mathbf{y}$-vertices and every vertex is either an $\mathbf{x}$-vertex or  $\mathbf{y}$-vertex, there exists a pair of adjacent vertices $(x^2, y^2)$ and $(x^3, y^3)$ such that $(x^2, y^2)$ is a $\mathbf{y}$-vertex and $(x^3, y^3)$ is an $\mathbf{x}$-vertex. As $y^1, y^2$ are vertices of $Q_B$, there exists a path from $y^1$ to $y^2$ in $Q_B$ with length at most $d(Q_B)$. If the path lifts successfully at $(x^1, y^1)$, then there exists $\hat{x}^1$ with $\s(\hat{x}^1) \subseteq \s(x^1)$ such that $(\hat{x}^1, y^2) \in P \oplus_2 Q$ is a vertex. Further, $d((x^1, y^1), (\hat{x}^1, y^2)) \leq d(Q_B)$. Next, we note that because $\hat{x}^1, x^2 \in P(y^2)$ are vertices, Lemma \ref{lem:m-path} implies that the distance between $(\hat{x}^1, y^2)$ and $(x^2, y^2)$ is at most $d(\bar{A})$. Because $(x^2, y^2)$ is adjacent to $(x^3, y^3)$, it follows that $d((x^1, y^1), (x^3, y^3)) \leq d(\bar{A}) + d(Q_B) + 1$.

If a step in the path from $y^1$ to $y^2$ fails to lift, then the vertex resulting from the first failed step is an $\mathbf{x}$-vertex. Thus, the distance between $(x^1, y^1)$ and an $\mathbf{x}$-vertex is at most $d(Q_B) \leq d(\bar{A}) + d(Q_B) + 1$.
\qed
\end{proof}

Lemma \ref{lem:change_cat} shows that every vertex is close to a vertex in the opposite category. For the remainder, we consider vertices in the same category. Now, we consider the case where $(x, y), (x', y')$ are both $\mathbf{x}$-vertices and $ax' \in \band_{\mathbf{x}}(y)$. In the proof of Lemma \ref{lem:in_band}, we explicitly construct a path between the two vertices; Figure \ref{fig:in_band} provides an illustration of this path. \newpage 

\begin{lem} \label{lem:in_band}
Let $(x, y), (x', y') \in P \oplus_2 Q$ be $\mathbf{x}$-vertices. If $ax' \in \band_{\mathbf{x}}(y)$, then 
\[
    d((x, y), (x', y')) \leq d(P_A) + 2d(\bar{A}) + d(\bar{B}) + 2.
\]
\end{lem}

\begin{proof}
Let $(x, y)$, $(x', y')$ be $\mathbf{x}$-vertices. Recall, $x, x' \in P_A$ are vertices. We let 
\[
x= x^0, \ldots, x^k =x'
\]denote a shortest path between $x$ and $x'$ in $P_A$ and note that $k \leq d(P_A)$. We consider several cases.

First, if the path lifts successfully to $(x, y)$, then there exists a vertex $(x^k, y^k)$ for some $y^k$ with $\s(y^k) \subseteq \s(y)$ and $d((x, y), (x^k, y^k)) = k \leq d(P_A)$. By Lemma \ref{lem:m-path}, we have that $d((x^k, y^k), (x^k, y')) \leq d(\bar{B})$. Thus, we have that 
\[
d((x, y), (x', y'))\leq d(P_A) + d(\bar{B}).
\]
In a following case, we use a nearly identical argument to establish a diameter bound; for reference, we label this argument ($\star$).

For the next case, suppose that a step fails to lift, and without loss of generality, we assume that the first step fails to lift. We recall that $y$ lies on an edge of $Q_B$, which we denote by $\operatorname{conv}(y^1, y^2)$.  By Lemma \ref{lem:band_step_criterion}, $ax^2 \notin \band_{\mathbf{x}}(y)$. First, we assume that $ax^2 > \alpha^* := \max \{ \alpha : \alpha \in \band_{\mathbf{x}}(y) \}$. By Lemma \ref{lem:out_of_band}, the step terminates at the point $(x^{1,2}, y^1)$ where $x^{1,2} \in \operatorname{conv}(x^1, x^2)$ and $ax^{1,2} = \alpha^* = c_a + c_b - by^1$ (relabeling $y^2$ as $y^1$ if necessary). Let $i$ be the largest index such that $ax^i > ax^{1,2}$. By Lemma \ref{lem:out_of_band}, there exists a unique point $x^{i,i+1} \in \operatorname{conv}(x^i, x^{i+1})$ such that $ax^{i,i+1} = ax^{1,2}$. By Lemma \ref{lem:m-path}, $d((x^{1,2}, y^1), (x^{i, i+1}, y^1)) \leq d(\bar{A})$, essentially, we ``jump.'' Next, we will lift steps continuing at $x^{i}$.  For reference later in the proof, we label this argument ($\star \star$).

If the step $y^1$ to $y^2$ lifts successfully to $(x^{i, i+1}, y^1)$, then the step terminates at $(x^{i, i+1}, y^2)$. Otherwise, the step terminates at $(x^{i+1}, y^{1,2})$ for some unique $y^{1,2} \in \operatorname{conv}(y^1, y^2)$. First, consider the case where the step terminates at $(x^{i+1}, \hat{y})$. If each step of $x^{i+1}, \ldots, x^k$ lifts successfully, then the walk is completed in an additional $k - (i+1) + d(\bar{B})$ steps according to ($\star$). Therefore, 
\[
d((x, y), (x',y')) \leq d_{P_A}(x, x') + d(\bar{A}) + d(\bar{B}) + 1.
\]

Now, suppose that a step in $x^{i+1}, \ldots, x^k$ fails to lift. Let $x^j$ to $x^{j+1}$ be the first step that fails. The step must terminate at $(x^{j,j+1}, y^2)$ (as opposed to $(x^{j,j+1}, y^1)$), for some unique $x^{j, j+1} \in \operatorname{conv}(x^j, x^{j+1})$, because $i$ was chosen maximally such that $ax^i > \alpha^*$. Therefore, $ax^{j+1} < \min \{ \alpha : \alpha \in \band_{\mathbf{x}}(y) \}.$ This case is identical to the case where the step terminated at $(x^{i, i+1}, y^1).$ We choose $\ell$ maximally such that $ax^{\ell} < ax^{j, j+1}.$ By Lemma \ref{lem:out_of_band}, there exists a unique point $x^{\ell,\ell+1}$ such that $(x^{\ell,\ell+1}, y^2)$ is a vertex. By ($\star \star$) $d((x^{j,j+1}, y^2), (x^{\ell,\ell+1}, y^2)) \leq d(\bar{A})$. From $(x^{\ell,\ell+1}, y^2)$, we lift the step $y^2$ to $y^1$, which terminates at $(x^{\ell+1}, y^{1,2})$ because $ax^{\ell+1} \in \band_{\mathbf{x}}(y)$ by maximality of $j$ and $\ell$. Further, each step of the remaining path $x^{\ell+1}, \ldots, x^k$ lifts successfully by maximality of $i$ and $\ell$. By ($\star$),
$d((x^{\ell+1}, y^{1,2}), (x', y')) \leq k - (\ell+1) + d(\bar{B}).
$

The case where the first step that fails to lift violates the minimum of $\band_{\mathbf{x}}(y)$ can be treated similarly. Thus, in all cases \[
d((x, y), (x', y')) \leq d_{P_A}(x, x') + 2d(\bar{A}) + d(\bar{B}) + 2. 
\]
\qed
\end{proof}

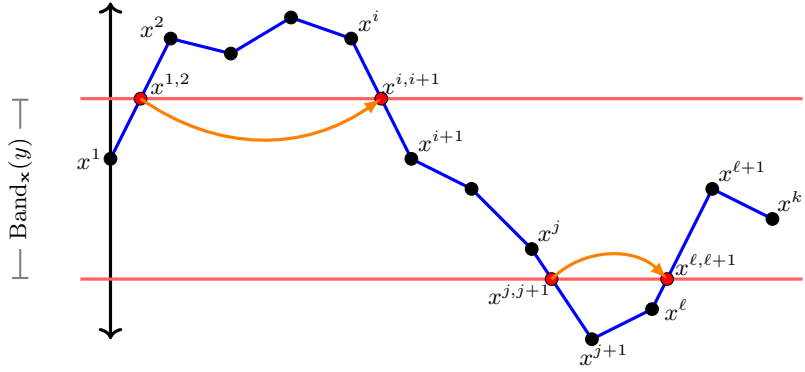
\begin{figure*} \label{band_path_fig}
    \centering
    \begin{tikzpicture}[scale = .8]
        \draw[very thick, <->] (0, -3) -- (0, 2.6);
        \draw[very thick, blue] (0, 0) -- (1, 2) -- (2, 1.75) -- (3, 2.35) -- (4, 2) -- (5, 0) -- (6, -0.5) -- (7, -1.5) -- (8, -3) -- (9, -2.5) -- (10, -0.5) -- (11, -1);
        \draw[very thick, pink!250] (-0.5, 1) -- (11.5, 1);
        \draw[very thick, pink!250] (-0.5, -2) -- (11.5, -2);
        
        \draw[fill = black] (0, 0) circle (3pt);
        \draw[fill = black] (1, 2) circle (3pt);
        \draw[fill = black] (2, 1.75) circle (3pt);
        \draw[fill = black] (3, 2.35) circle (3pt);
        \draw[fill = black] (4, 2) circle (3pt);
        \draw[fill = black] (5, 0) circle (3pt);
        \draw[fill = black] (6, -0.5) circle (3pt);
        \draw[fill = black] (7, -1.5) circle (3pt);
        \draw[fill = black] (8, -3) circle (3pt);
        \draw[fill = black] (9, -2.5) circle (3pt);
        \draw[fill = black] (10, -0.5) circle (3pt);
        \draw[fill = black] (11, -1) circle (3pt);
        \draw[fill = red] (0.5, 1) circle (3pt);
        \draw[fill = red] (4.5, 1) circle (3pt);
        \draw[fill = red] (7.33, -2) circle (3pt);
        \draw[fill = red] (9.25, -2) circle (3pt);

        \draw[very thick, orange, -latex] (0.5, 1) to[bend left = -35] (4.5, 1);
        \draw[very thick, orange, -latex] (7.33, -2) to[bend left = 45] (9.25, -2);

        \node at (-.35, 0) {$x^1$};
        \node at (0.7, 2.2) {$x^2$};
        \node at (4.3, 2.3) {$x^i$};
        \node at (5.5, 0.3) {$x^{i+1}$};
        \node at (7.3, -1.2) {$x^j$};
        \node at (8.2, -3.25) {$x^{j+1}$};
        \node at (9.4, -2.5) {$x^{\ell}$};
        \node at (10.5, -0.2) {$x^{\ell+1}$};
        \node at (11.3, -0.7) {$x^k$};
        \node at (1, 1.25) {$x^{1,2}$};
        \node at (5, 1.25) {$x^{i,i+1}$};
        \node at (6.8, -2.25) {$x^{j,j+1}$};
        \node at (9.9, -1.75) {$x^{\ell, \ell+1}$};

        \node [rotate=90] at (-1.5, -0.5) {$\band_{\mathbf{x}}(y)$};
        \draw[thick, gray, -|] (-1.5, 0.5) -- (-1.5, 1);
        \draw[thick, gray, -|] (-1.5, -1.5) -- (-1.5, -2);
    \end{tikzpicture}
    \caption{Illustration of the path constructed in Lemma \ref{lem:in_band}.}
    \label{fig:in_band}
\end{figure*}

We are ready to prove our bound for $d(P \oplus_2 Q)$.

\begin{thm} \label{thm:2sum_diam_bound}
The diameter of $P \oplus_2 Q$ is at most 
\[
d(A)(d(\bar{A}) + 1) + d(B)(d(\bar{B}) + 1) 
+ \min\{ d(\bar{A}) + d(B) + 1, d(A) + d(\bar{B}) + 1 \}.
\]
\end{thm}

\begin{proof}
Let $(x, y), (x', y') \in P \oplus_2 Q$ be $\mathbf{x}$-vertices. The case with two $\mathbf{y}$-vertices follows from interchanging the roles of $x$ and $y$. The case with one $\mathbf{x}$-vertex and one $\mathbf{y}$-vertex follows from Lemma \ref{lem:change_cat} by noting that the $\mathbf{y}$-vertex is at most $d(\bar{A}) + d(B) + 1$ steps from an $\mathbf{x}$-vertex.

We assume $ax' \notin \band_{\mathbf{x}}(y)$ and $ax \notin \band_{\mathbf{x}}(y')$; otherwise, we are done by Lemma \ref{lem:in_band}. Further, we assume that $ax' > \max \{ \alpha : \alpha \in \band_{\mathbf{x}}(y) \}$ and $ax < \min \{ \alpha : \alpha \in \band_{\mathbf{x}}(y') \}$. Given a vertex $(x, y)$ and some other $x'$, we define $y_*$ to be the unique point satisfying $by_* = c_a + c_b - ax'$ and $\s(y_*) \subseteq \s(y)$. We define $x_*$ analogously.

Let $x = x^0, x^1, \ldots, x^k = x'$ denote a shortest path between $x$ and $x'$ in $P_A$. Let $i$ denote the largest index such that $ax^{i} \in \band_{\mathbf{x}}(y)$. By Lemma \ref{lem:in_band}, $(x^{i}, y_*) \in P \oplus_2 Q$ is a vertex and 
\[
d((x^0, y), (x^{i}, y_*)) \leq i + 2d(\bar{A}) + 1.
\]
By maximality of $i$, the step $x^{i}$ to $x^{i+1}$ fails to lift and terminates at $(x^{i, i+1}, y^1)$ where $y^1$ is a vertex on the edge of $Q_B$ containing $y$ and $x^{i, i+1} \in \operatorname{conv}(x^{i}, x^{i+1})$ such that $ax^{i, i+1} = c_a+c_b-by^1$. 

Without loss of generality, we assume that $ax^{i, i+1} = \max \{ \alpha : \alpha \in \band_{\mathbf{x}}(y) \}$. Consider the vertex $(x^k, y')$. If $ax^{i} \in \band_{\mathbf{x}}(y')$, then $d((x^k, y'), (x^i, y_*')) \leq k-i + 2d(\bar{A}) + 1$. Further, by Lemma \ref{lem:m-path}, we have $d((x^i, y_*), (x^i, y_*')) \leq d(\bar{B})$. By combining these bounds, we are done. 

Thus, we now assume $ax^{i} \notin \band_{\mathbf{x}}(y')$. Let $j \in \{ i+1, \ldots, k \}$ be the smallest index such that $ax^{j} \in \band_{\mathbf{x}}(y')$, which exists because $ax^k \in \band_{\mathbf{x}}(y')$. By Lemma \ref{lem:in_band}, there exists $y'_*$ with $\s(y'_*) = \s(y')$ and 
\begin{equation}\label{eqn:x^ky'}
d((x^k, y'), (x^{j}, y'_*)) \leq k-j + 2d(\bar{A}) + 1.
\end{equation}

The step $x^{j}$ to $x^{j-1}$ fails to lift, and the resulting vertex is $(x^{j-1, j}, (y')^1)$, where $(y')^1$ is the vertex on the edge of $Q$ containing $y'$ and $x^{j-1, j} \in \operatorname{conv}(x^{j-1}, x^{j})$ such that $ax^{j-1, j} = c_a+c_b-b(y')^1$. Without loss of generality, we assume that $ax^{j-1, j} = \min \{ \alpha : \alpha \in \band_{\mathbf{x}}(y') \}$.

Recall that $y^1, (y')^1 \in Q_B$ are vertices. Consider a shortest path $y^1$, $y^2$, $\ldots$, $y^r = (y')^1$. We lift steps from the path to $(x^{i, i+1}, y^1)$ until a step fails to lift. There are three possible outcomes:\\ 

\begin{enumerate}
    \item Each step lifts successfully; the walk terminates at $(x^{i, i+1}_*, y^r)$.
    \item Each step up to some $\ell$ lifts successfully; the walk terminates at $(x^{i}, y^{\ell, \ell +1})$.
    \item Each step up to some $\ell$ lifts successfully; the walk terminates at $(x^{i+1}, y^{\ell, \ell +1})$.
\end{enumerate}
In the first case, note that by Lemma \ref{lem:m-path} we have
\[
d((x^{i, i+1}_*, y^r), (x^{j-1,j}, y^r)) \leq d(\bar{A}).
\]
Further, recall $(x^{j-1,j}, y^r)$ is adjacent to $(x^j, y_*')$. By Equation \ref{eqn:x^ky'}, we have a bound on the number of steps in the remainder of the walk. Therefore, we are done.

It remains to resolve the second and third case, i.e., the cases in which a step fails to lift. Note that when the walk reaches a vertex $(x^{j+1}, \hat{y})$ or $(\hat{x}, y^r)$ for some $\hat{x}$ or $\hat{y}$, we can complete it as explained above. We will show that, before we reach such a vertex, we can iteratively move ``closer'': in a certain number of steps from a vertex of form $(x^t, y^{u,u+1})$ for $i \leq t < j+1$ and $\ell \leq u < r$, we can either reach the vertex $(x^{t+1}, y^{u',u'+1})$ for $u \leq u'$ or reach the vertex $(x^{t,t+1}, y^r)$. We first only prove reachability; we later conclude the argument with a bound on the number of steps.

Consider the second case: the step $y^{\ell}$ to $y^{\ell+1}$ terminates at $(x^{i}, y^{\ell, \ell +1})$. If $ax^j \in \band_{\mathbf{x}}(y^{\ell, \ell+1})$, then by Lemma \ref{lem:in_band} we have 
\[
    d((x^{i}, y^{\ell, \ell +1}), (x^{j}, y'_*)) \leq 
    d(P_A) + 2d(\bar{A}) + d(\bar{B}) + 2,
\]
and, in combination with Equation \ref{eqn:x^ky'}, we are done.
Now suppose that $ax^j \notin \band_{\mathbf{x}}(y^{\ell, \ell+1})$. 
There exists a maximal $s \in \{\ell+1, \ldots, r \}$ such that $by^s > by^{\ell, \ell +1}$, because $by^{\ell +1} > by^{\ell, \ell +1} > by^r$. Further, there exists a unique $y^{s, s+1} \in \operatorname{conv}(y^s, y^{s+1})$ such that $by^{s, s+1} = by^{\ell, \ell +1}$ and $(x^i, y^{s, s+1})$ is a vertex. By Lemma \ref{lem:m-path}, we have
$
d((x^{i}, y^{\ell, \ell +1}), (x^{i}, y^{s, s+1})) \leq d(\bar{B})$.

Next, consider the possible lift of a step from $x^i$ to $x^{i+1}$ at $(x^{i}, y^{s, s+1})$. If $ax^{i+1} \in \band_{\mathbf{x}}(y^{s, s+1})$, then the vertices $(x^{i}, y^{s, s+1})$ and $(x^{i+1}, y^{s, s+1})$ are adjacent. Thus, 
\[
d((x^{i}, y^{\ell, \ell +1}), (x^{i+1}, y^{s, s+1})) \leq 1+d(\bar{B}).
\]
If $ax^{i+1} \notin \band_{\mathbf{x}}(y^{s, s+1})$, then the step from $x^i$ to $x^{i+1}$ fails to lift and terminates at some $(x^{i, i+1}_*, y^{s+1})$. We proceed by lifting steps of the path $y^{s+1}, \ldots, y^r$ starting from $(x^{i,i+1}_*, y^{s+1})$. By maximality of $s$, the path lifts successfully and terminates at $(x^{i,i+1}_*, y^r)$, or some step fails to lift and terminates at $(x^{i+1}, y^{t, t+1})$ for some $t \geq s+1$. Thus, we have reached a vertex with component $x^{i+1}$ or $y^r$.

The third case works analogously. For either case, we can iterate the process described above to reach a vertex of form $(x^{j+1}, \hat{y})$ or $(\hat{x}, y^r)$.
It remains to devise a bound on the number of steps. Starting from vertex $(x^i, y^1)$, there exists a $\hat{y}$ such that $(x^{i+1}, \hat{y})$ is a vertex and $d((x^i, y), (x^{i+1}, \hat{y}))$ is at most $1 + d(\bar{A})$ plus at most an additional $1+d(\bar{B})$ steps for each lifted step in the walk $y^1, \ldots, y^r$: 
\[
d((x,y), (x',y')) \leq d(A)(1 + d(\bar{A})) + d(B)(1 + d(\bar{B})).
\]
Finally, note that an additional $d(\bar{A}) + d(B) + 1$ or $d(A) + d(\bar{B}) + 1$ steps may be required to walk to an $\mathbf{x}$-vertex or $\mathbf{y}$-vertex, respectively.
\qed
\end{proof}

We have shown that the diameter of $P \oplus_2 Q$ is at most quadratic in $d(\bar{A})$, $d(A)$, $d(\bar{B})$, $d(B)$. With some additional arguments, our work may lead to a diameter bound that is linear in $d(\bar{A})$ and $d(\bar{B})$. 

\begin{rmk}
Our diameter bound for the $2$-sum is linear if the following is true. Let $P \subseteq \R^{n}$ be a simple polyhedron with vertices $v$, $w$. Let $v = v^0, v^1, \ldots, v^k = w$ denote a shortest path from $v$ to $w$. Suppose that there exists $(a, \gamma) \in \R^{n+1}$ such that $av^0 < \gamma < av^1$ and  $av^{k-1} > \gamma > av^k$, then there exists $v^{0,1} \in \operatorname{conv}(v^0, v^1)$ and $v^{k-1,k} \in \operatorname{conv}(v^{k-1}, v^k)$ such that $av^{0,1} = av^{k-1, k} = \gamma$. Let $P' = P \cap \{ x : ax = \gamma \}$. It follows that $av^{0,1}, av^{k-1, k} \in P'$ are vertices. Is $d(av^{0,1}, av^{k-1, k}) \in O(k)$?
\end{rmk}

In our work, the best bound we have is $d(\bar{A})$, which leads to a quadratic bound. We note that the simplicity of $P$ is necessary; when the polyhedron is degenerate, even in dimension $3$, $d(\hat{v}^1, \hat{v}^{k-1})$ can be arbitrarily larger than $k$. For example, consider a 3-dimensional pyramid with an $n$-gon base. This polyhedron has diameter 2, however, a slice of the polyhedron could have diameter $n/2$.

\subsection{Application to Appending a Unit Column}

The addition of a unit column to a matrix is another operation relevant to TU polyhedra; it, too, appears in Seymour's decomposition \cite{s-98}. Further, appending a unit column to the equality constraint matrix of a polyhedron is equivalent to relaxing an equality constraint to inequality. We can view this operation as a special case of the $2$-sum. Observe that, for a matrix $M$, $\begin{bmatrix} 1 & 1 \end{bmatrix} \oplus_2 M$ prepends a unit column to the matrix. While Theorem \ref{thm:2sum_diam_bound} gives a quadratic bound in general, we adjust our methods to yield a linear bound for this case. Consider the following two polyhedra:
\[
P := \left\{ x : \begin{bmatrix}   
    A \\
    a
\end{bmatrix}
x
= 
\begin{bmatrix}   
    c_A \\
    c_a
\end{bmatrix},
x \geq 0
\right\}, \quad
P_{\bar{A}} := \left\{ \begin{bmatrix}   
    x\\
    s
\end{bmatrix} : \begin{bmatrix}   
    A & 0 \\
    a & 1
\end{bmatrix}
\begin{bmatrix}   
    x\\
    s
\end{bmatrix}
= 
\begin{bmatrix}   
    c_A \\
    c_a
\end{bmatrix},
x, s \geq 0
\right\}.
\]
We will prove that the diameter of $P_{\bar{A]}}$ is bounded above by a linear function in terms of $d(P)$ and $d(A)$. Let $P_A$ be the same as before. A bound for $d(P_{\bar{A}})$ also yields a diameter bound for the polyhedron arising from $P_A$ by relaxing $ax = c_a$ to $ax \leq c_a$.

\begin{thm}
The diameter of $P_{\bar{A}}$ is at most $
d(A) + d(P) + 2$.
\end{thm}

\begin{proof}
Suppose that $(x, 0), (x', 0)$ are vertices for which $s$ is nonbasic. Note that $x, x'$ are vertices of $P$, and $P$ is a face of $P_{\bar{A}}$. Thus, 
\[
d_{P_{\bar{A}}}((x, 0), (x', 0)) \leq d(P).
\]

Now, suppose that $(x, s), (x', s')$ are vertices where $s$ and $s'$ are basic variables. Note that $x, x' \in P_A$ are vertices. We consider a shortest path $x = x^1, x^2, \ldots, x^k = x'$, and note that $k \leq d(P_A)$. If for each $x^i$, $ax^i \leq c_a$, then the walk lifts successfully. Otherwise, there exists a minimal $i$ and maximal $j$ such that $ax^{i+1} > c_a$, $ax^j > c_a$. The step $x^{i}$ to $x^{i+1}$ fails to lift to $P_{\bar{A}}$ and results in the vertex $(x^{i, i+1}, 0)$ for some unique $x^{i, i+1} \in \operatorname{conv}(x^i, x^{i+1})$ with $ax^{i, i+1} = c_a$. Likewise, there exists a unique $x^{j, j+1} \in \operatorname{conv}(x^j, x^{j+1})$ with $ ax^{j, j+1} = c_a$. It follows that $(x^{i, i+1}, 0)$ and $(x^{j, j+1}, 0)$ are vertices of $P_{\bar{A}}$. Further, $d((x^{i, i+1}, 0),(x^{j, j+1}, 0)) \leq d(P)$. Thus, 
\[
d_{P_{\bar{A}}}((x^1, s^1), (x^2, s^2)) \leq d(P) + d(P_{\bar{A}}) + 1.
\]

Finally, suppose that $(x, s), (x', s')$ are a pair of vertices where $s$ is nonbasic and $s'$ is basic. The vertex $(x, s)$ is adjacent to a vertex for which $s$ is basic, and thus 
\[
d_{P_{\bar{A}}}((x, s), (x', s')) \leq d(P) + d(P_{\bar{A}}) + 2.
\]
\qed
\end{proof}

\section{Applications to the $3$-Sum Polyhedron}\label{sec:3sum}

In addition to the $2$-sum, Seymour's decomposition of TU matrices makes use of the $3$-sum, defined as follows 
\[
    \begin{bmatrix}
        A & a & a \\
        c & 0 & 1
    \end{bmatrix}
    \oplus_3
    \begin{bmatrix}
        1 & 0 & b \\
        d & d & B
    \end{bmatrix}
    =
    \begin{bmatrix}
        A & ab \\
        dc & B \\
    \end{bmatrix}
\]
where $a$ and $d$ are column vectors of appropriate size, $b$ and $c$ are row vectors of appropriate size.
Like the $2$-sum, the $3$-sum arose as a matroid operation and also is defined for graphs and matrices \cite{s-98}. For graphs, the $3$-sum corresponds to the $3$-clique sum. We define the $3$-sum for standard-form polyhedra by applying the $3$-sum operation to the constraint matrices and combining the right-hand sides in a particular way.

\begin{defi}\label{def:3sum}
Given two polyhedra $P$ and $Q$ in the following forms
\begin{align*}
    P &:= \left\{ x : \begin{bmatrix}
        A & a & a \\
        c & 0 & 1
    \end{bmatrix}
    \begin{bmatrix}
        x \\
        s_1 \\
        s_2
    \end{bmatrix}
    =
    \begin{bmatrix}
        c_A \\
        c_a
    \end{bmatrix}, \, 
    x, s_1, s_2 \geq 0
    \right\}, \\
    Q &:= \left\{ y : \begin{bmatrix}
        1 & 0 & b \\
        d & d & B
    \end{bmatrix}
    \begin{bmatrix}
        s_1 \\
        s_2 \\
        y
    \end{bmatrix}
    =
    \begin{bmatrix}
        c_b \\
        c_B
    \end{bmatrix}, \, 
    y, s_1, s_2 \geq 0
    \right\},
\end{align*}
the $3$-sum of $P$ and $Q$ is 
\[
    P \oplus_3 Q:= \left\{ \begin{bmatrix}
        x \\
        y
    \end{bmatrix} : \begin{bmatrix}
        A & ab \\
        dc & B \\
    \end{bmatrix}
    \begin{bmatrix}
        x \\
        y
    \end{bmatrix}
    =
    \begin{bmatrix}
        c_A + ac_b \\
        c_B + dc_a
    \end{bmatrix}, \, 
    x, y \geq 0
    \right\}.
\]
\end{defi}

The right-hand sides are chosen carefully to ensure that vertices $(x,y)$ of $P \oplus_3 Q$ project to vertices of the input.
Mirroring the $2$-sum, when $ab \neq 0$, $dc \neq 0$, we can row-reduce $P \oplus_3 Q$ to the following form; for brevity, we omit the details

\[
    \left\{ \begin{bmatrix}
        x \\
        y
    \end{bmatrix} : \begin{bmatrix}
        A & 0 \\
        a_1 & b_1 \\
        a_2 & b_2 \\
        0 & B \\
    \end{bmatrix}
    \begin{bmatrix}
        x \\
        y
    \end{bmatrix}
    =
    \begin{bmatrix}
        c_A \\
        c_{a_1} + c_{b_1} \\
        c_{a_2} + c_{b_2} \\
        c_B 
    \end{bmatrix}, \\
    x, y \geq 0
    \right\}.
\]

Next, we also extend the definition of the $\mathbf{x}$-band to the $3$-sum. For the $3$-sum, the $\mathbf{x}$-band is a 2-dimensional polytope because of the two shared constraints, and thus bounds for the $2$-sum do not immediately transfer. However, we can prove distance bounds for \emph{some} pairs of vertices of the $3$-sum.

Like the $2$-sum, we can categorize the vertices of $P \oplus_3 Q$. We say that a vertex $(x, y) \in P \oplus_3 Q$ is an \emph{$\mathbf{x}$-vertex} if $x \in P_A$ is a vertex; likewise, we say $(x, y)$ is a \emph{$\mathbf{y}$-vertex} if $y \in Q_B$ is a vertex. For the $3$-sum, there is an additional category: we say that $(x, y)$ is a \textit{mixed-vertex} if $x$ lies on an edge of $P_A$ and $y$ lies on an edge of $Q_B$. 

\begin{defi}\label{def:band_3sum}
Let $(x, y) \in P \oplus_3 Q$ be a vertex. The $\mathbf{x}$-band of $(x, y)$ is the following subset of $\R^2$:
\begin{align*}
    \band_{\mathbf{x}}(y) := \{ (c_a^1 + c_b^1 - b_1z, \, c_a^2 + c_b^2 - b_2z): 
    Bz = c_B, \, \s(z) \subseteq \s(y), \, z \geq 0 \}.
\end{align*}
\end{defi}

The $\mathbf{x}$-band for the $3$-sum still yields a criterion to determine when two adjacent vertices in $P_A$ correspond to adjacent vertices in $P \oplus_3 Q.$ We explain this criterion in Lemma \ref{lem:3sum_band_lem} which can be proved by modifying the proof of Lemma \ref{lem:band_step_criterion}.

\begin{lem}\label{lem:3sum_band_lem}
    Let $(x, y) \in P \oplus_3 Q$ be an $\mathbf{x}$-vertex, and let $x' \in P_A$ be an adjacent vertex to $x$. The point $(x', y^*)$ is an adjacent vertex of $(x, y)$ in $P \oplus_3 Q$ for some $y^*$ with $\s(y^*) \subseteq \s(y)$ if and only if $(a_1x', a_2x') \in \band_{\mathbf{x}}(y)$.
\end{lem}

Like the $2$-sum, the distance between $\mathbf{x}$-vertices with identical $x$-coordinates is small. Lemma \ref{lem:m-path_3sum} generalizes Lemma \ref{lem:m-path} to the $3$-sum. A proof mirrors the proof of Lemma \ref{lem:m-path}.

\begin{lem}\label{lem:m-path_3sum}
   If $(x,y), (x,y') \in P \oplus_3 Q$  are both $\mathbf{x}$-vertices, then 
    \[
        d((x,y), (x,y')) \leq d(\bar{B}).
    \]
\end{lem}

Lemma \ref{lem:m-path_3sum} enables us to prove a $3$-sum analogue of Lemma \ref{lem:in_band} for $\mathbf{x}$- and $\mathbf{y}$-vertices. There is a remaining challenge in determining the distance between pairs of mixed-vertices. The $\mathbf{x}$-band is an $n$-gon because it is a bounded $2$-dimensional projection of a polyhedron. In some cases, when two mixed-vertices lie on the same edge of the $\mathbf{x}$-band, the distance between them can be bounded by the diameter of a $2$-sum. In general, however, the arguments we present do not give a distance bound for pairs of mixed-vertices that are not on the same band.

It is plausible that an understanding of the distance between mixed-vertices will yield a diameter bound for the $3$-sum. Following the strategy of Theorem \ref{thm:2sum_diam_bound}, it would be natural to show that any vertex is close to a mixed-vertex and then construct a walk between a pair of mixed-vertices.

\section{Conclusion and Outlook}\label{sec:outlook}

In this paper, we established that the diameter of a $2$-sum polyhedron is quadratic in the diameters of its parts. The introduction of the band of a vertex provides a framework that relates adjacency in $P \oplus_2 Q$ to adjacency in $P$ and $Q$, enabling walks to be lifted from $P$ and $Q$ to $P \oplus_2 Q$. In fact, we proved that the distance between pairs of vertices in the same band is linear; essentially one has linear diameter bounds within certain sections of $P \oplus_2 Q$. It remains of high interest to study whether a general linear bound can be derived. 

If the diameter of the $2$-sum is linear, then it follows that the diameter of a TU polyhedron whose constraint matrix is not a $3$-sum is a linear function of the Hirsch bound. Our proof technique involved ``jumps,'' each bounded by the diameter of the underlying polyhedra. We showed that if a stronger bound on the number of steps in these jumps can be derived, then the $2$-sum diameter bound may become linear. 

Other natural directions of research include investigating the diameter of an iterated $2$-sum with the same base polyhedron, i.e., $P \oplus_2 P \oplus_2 \cdots \oplus_2 P$. In particular, it would be interesting to determine whether the bound on the diameter can be improved by exploiting the special structure of the construction. Additionally, our bound relates the diameter of a $2$-sum polyhedron to the class of polyhedra that allow the decomposition. Given a $2$-sum polyhedron $R$ and polyhedra $P$ and $Q$ such that $R = P \oplus_2 Q$, can we bound the diameter of $R$ in terms of just the diameters of $P$ and $Q$? In general, the answer is no. For any $n \in \N$, we can find $R$, $P$, and $Q$ such that $R = P \oplus_2 Q$, $\dim(R) = n$ and $\dim(P) = \dim(Q) = 0$. However, if we require that  $\dim(R) = \dim(P) + \dim(Q)$, essentially $P$ and $Q$ are not degenerate and are given by minimal descriptions, then it remains open whether or not a quadratic bound in terms of $d(P)$ and $d(Q)$ exists.

Further, our techniques transfer to the addition of a unit column and certain faces of $3$-sum polyhedra. Adding a unit column is a special case of the $2$-sum operation and thus, our bound transfers immediately. In fact, through a careful analysis, we proved that the addition of a unit column retains a linear diameter bound. For the $3$-sum, we demonstrated that the distance between vertices on certain faces is also linear. We believe that the band of a vertex, which we showed in the $3$-sum setting to take the form of a convex polytope in dimension $2$, can be a powerful tool to establish a polynomial bound on the diameters of general $3$-sums. In our construction of the $2$-sum diameter bound, when a step left the band, we were able to construct a jump of bounded length to re-enter the band. Finding such a jump for the $3$-sum is more challenging and would be a crucial step in using the band of a vertex to construct a diameter bound for the $3$-sum. A promising starting point is to determine how to lift walks from the input polyhedra to construct walks between mixed-vertices.

\section*{Acknowledgment}

\noindent This work was supported by Air Force Office of Scientific Research grant FA9550-21-1-0233 (Complex Networks) and NSF grant 2006183 (Algorithmic Foundations, Division of Computing and Communication Foundations).

\end{document}